\documentclass{amsart}

\usepackage{amssymb,amsmath,array,multirow,makecell,blindtext,amsthm,enumitem,mathtools,amscd,tikz-cd,amsrefs,dsfont,hyperref,url,caption,float,placeins}
\usepackage[font=small, labelfont=bf]{caption}
\hypersetup{
    colorlinks=true,
    linkcolor=blue,
    citecolor=black,
    filecolor=magenta,      
    urlcolor=cyan,
}

\usepackage[all]{xy}
\usepackage{graphicx}
\usetikzlibrary{positioning}
\usepackage{caption}

\theoremstyle{plain}
\newtheorem{theorem}{Theorem}[section]
\newtheorem{corollary}[theorem]{Corollary}
\newtheorem{lemma}[theorem]{Lemma}
\newtheorem{remark}[theorem]{Remark}
\newtheorem{proposition}[theorem]{Proposition}
\newtheorem{definition}[theorem]{Definition}
\newtheorem{question}[theorem]{Question}

\newtheorem{example}[theorem]{Example}

\newcommand{\Z}{\mathbb{Z}}
\newcommand{\Q}{\mathbb{Q}}

\newcommand{\Oo}{\mathcal{O}}
\newcommand{\Gg}{\mathcal{G}}

\newcommand{\Gal}{\operatorname{Gal}}
\newcommand{\SL}{\operatorname{SL}}
\newcommand{\sgn}{\operatorname{sgn}}
\newcommand{\Ent}{\operatorname{Ent}}
\newcommand{\ord}{\operatorname{ord}}
\newcommand{\Aut}{\operatorname{Aut}}
\newcommand{\im}{\operatorname{im}}

\newcommand{\GL}{\operatorname{GL}}
\newcommand{\nf}{\normalfont}
\newcommand{\Mod}[1]{\ \mathrm{mod}\ #1}

\title[Entanglement of elliptic curves upon base extension]{Entanglement of elliptic curves upon \\ base extension}

\author{Tori Day and Rylan Gajek-Leonard}

\address[]{Department of Mathematics, Mount Holyoke College, South Hadley, MA}
\email[Tori Day]{tori.day@mtholyoke.edu}

\address[]{Department of Mathematics, Union College, Schenectady, NY}
\email[Rylan Gajek-Leonard]{gajekler@union.edu}


\begin{document}

\begin{abstract}   Fix distinct primes $p$ and $q$ and let $E$ be an elliptic curve defined over a number field $K$. The $(p,q)$-entanglement type of $E$ over $K$ is the isomorphism class of the group $\Gal(K(E[p])\cap K(E[q])/K)$. The size of this group measures the extent to which the image of the mod $pq$ Galois representation attached to $E$ fails to be a direct product of the mod $p$ and mod $q$ images. In this article, we study how the $(p,q)$-entanglement group varies over different base fields. We prove that for each prime $\ell$ dividing the greatest common divisor of the size of the mod $p$ and $q$ images, there are infinitely many fields $L/K$ such that the entanglement over $L$ is cyclic of order $\ell$. We also classify all possible $(2,q)$-entanglement types that can occur as the base field $L$ varies. 
\end{abstract}

\maketitle
%\tableofcontents

%%%%%%%%
\section{Introduction}

Let $E$ be an elliptic curve defined over a number field $K$. For any integer $n\geq 2$,  the action of $G_K=\Gal(\bar K /K)$ on the $n$-torsion points $E[n]$ yields a representation
$$
\rho_{E,n}:G_K \rightarrow \Aut(E[n])\cong \GL_2(\Z/n\Z)
$$
whose image is well-defined up to conjugation. A classical problem, often referred to as Mazur's `Program B' \cite{Maz77}, is to determine all possible images of this representation.  When $n=p$ is prime and $E$ does not have complex multiplication, a well-known theorem of Serre \cite{Ser72} asserts that $\rho_{E,p}$ is surjective for all but finitely many \emph{exceptional} primes $p$. Recent work of Zywina \cite{Zyw15} classifies the possible images of $\rho_{E,p}$ for rational elliptic curves at exceptional primes. 

When $n$ is composite, the problem of classifying images becomes even more subtle. This is due to the fact that it is not always possible to determine the image mod $n$ solely in terms of the images modulo the prime factors of $n$. In general, one only has an injection 
$$
\im\rho_{E,n}\hookrightarrow\prod_{p\mid n}\im\rho_{E,p^{\ord_p(n)}}.
$$
The failure of surjectivity in the above map can be attributed to coincidental intersections, or \emph{entanglements}, between various torsion fields at primes (or prime powers) dividing $n$. 

For example, if $n=pq$ is the product of distinct primes $p$ and $q$ then the index of $\im\rho_{E,pq}$ inside $\im\rho_{E,p}\times\im\rho_{E,q}$ is equal to the degree
$$
[K(E[p])\cap K(E[q]):K],
$$
where the $p$ and $q$-torsion fields $K(E[p])$ and $K(E[q])$ are (by definition) those fixed by the kernels of $\rho_{E,p}$ and $\rho_{E,q}$, respectively. If this degree is nontrivial, $E$ is said to have \emph{$(p,q)$-entanglement over $K$}. 
Thus, $(p,q)$-entanglements occur precisely when the mod $pq$ image fails to be the full direct product of the mod $p$ and mod $q$ images.  The entanglement \emph{type} is the isomorphism class of the Galois group 
$$
\Gal(K(E[p])\cap K(E[q])/K),
$$
which we denote $T_{p,q}(E/K)$. Knowing the entanglement type, together with the mod $p$ and mod $q$ image, allows one to compute the full image mod $pq$.

 In this article, we consider the behavior of entanglements along base extensions. Specifically, we ask the following question:
\begin{question}\label{OurQ} \it 
For a fixed elliptic curve $E/K$ and distinct primes $p$ and $q$, what groups $T_{p,q}(E/L)$ occur as we vary over base extensions $L/K$?\nf
\end{question}
 
\noindent  Our focus is therefore on the set 
\begin{equation}\label{entset}
\Ent_{p,q}(E/K)=\{T_{p,q}(E/L) \mid \text{$L/K$} \}.
\end{equation}
This set is finite and consists of finite groups (even when $L/K$ is an infinite extension), since for any $L/K$ we have (see Proposition \ref{sizebound}) 
 \begin{equation}\label{div1}
 \#T_{p,q}(E/L)\,\bigg|\,\, \gcd(\#\im\rho_{E,p},\#\im\rho_{E,q}).
 \end{equation}
It is easy to see that $\Ent_{p,q}(E/K)$ always contains the trivial group: if $L$ is any field containing $K(E[pq])$ then $T_{p,q}(E/L)=1$. For arbitrary $L/K$, the possibilities for $(p,q)$-entanglement are less obvious. (For example, the entanglement can both increase and decrease in size over different base fields -- see Examples \ref{quadex} and Examples \ref{ex-2in3}).
To ease notation, we define
$$ d(E,p,q)=\gcd(\#\im\rho_{E,p},\#\im\rho_{E,q}).
$$
Our main theorem, stated as follows (see also Theorem \ref{main1}), shows that in addition to the trivial group, the set $\Ent_{p,q}(E/K)$ contains
all cyclic groups of prime order dividing $ d(E,p,q)$, i.e., it contains the maximal number of prime-order cyclic groups permissible by \eqref{div1}. 

\begin{theorem}\label{main} 
Let $\ell$ be a prime dividing $d(E,p,q)$. Then there exists a field $L\subseteq K(E[pq])$ such that
$$
T_{p,q}(E/L)= \Z/\ell\Z.
$$
\end{theorem}

For an extension $L/K$, let $\rho_{E/L,n}$ denote the restriction of $\rho_{E,n}$ to the subgroup $\Gal(\bar K/L)\subseteq G_K$. Theorem \ref{main} immediately implies the following corollary.

\begin{corollary} For each $\ell$ as in Theorem \ref{main}, there is a field $L/K$ such that the image of $\rho_{E/L,pq}$ contains a normal subgroup of index $\ell$. 
\end{corollary}

\begin{example}\label{quadex}  \nf  Let $E/\Q$ be the elliptic curve with Cremona label \href{https://www.lmfdb.org/EllipticCurve/Q/1323/e/1}{\texttt{1323p1}}. One can check that $\rho_{E,2}$ is surjective while $\im\rho_{E,3}$ has index 3 in $\GL_2(\Z/3\Z)$, hence 
$$
d(E,2,3)=2.
$$
It follows from \eqref{div1} and Theorem \ref{main} that the only $(2,3)$-entanglement types that occur for $E$ (over any base field) are trivial entanglement and $\Z/2\Z$-entanglement, i.e., 
$$
\Ent_{2,3}(E/\Q)=\{1,\Z/2\Z\}.
$$
 In fact, by using the subfield structure of $\Q(E[2])$ and $\Q(E[3])$ (see the diagram below), we see that
 $$
 T_{2,3}(E/\Q)=1\qquad \text{and}\qquad T_{2,3}(E/L)=\Z/2\Z,
 $$
 for $L\in \{\Q(\sqrt{-1}),\Q(\sqrt{3}),\Q(\sqrt{7})\}$.
 The rightmost diagram, wherein $L=\Q(\sqrt{-1})$, illustrates the idea that entanglement can increase in size by `glueing' together various subfields of the $p$ and $q$-torsion fields. This idea is explored in detail in \S\ref{entfieldsec}. Note also that there are \emph{multiple} fields $L\subseteq \Q(E[6])$ which achieve $\Z/2\Z$-entanglement.  This has to do with the fact that the prime $q=3$ is exceptional for $E$ -- see Remark \ref{quadexfull} for a more complete discussion.  
 
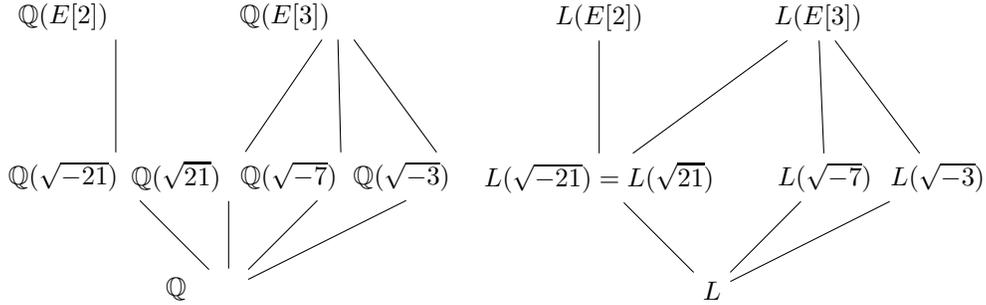
\begin{minipage}{.15\textwidth}
\begin{tikzpicture}[node distance = 1.5cm, auto]
\hspace{-0.7cm}
     \node (1) {$\Q(E[2])$};
     \node (2)[right=of 1]   {$\Q(E[3])$};
     \node (3)[below=of 1]  {$\Q(\sqrt{-21})$};
     \node (4)[right of=3]  {$\Q(\sqrt{21})$};
     \node (5)[right of=4]  {$\Q(\sqrt{-7})$};
     \node (6)[right of=5]  {$\Q(\sqrt{-3})$};
      \node (7)[below of=3, right of = 3]  {$\Q$};
     \draw[-] (1) to node {} (3);
     \draw[-] (2) to node {} (4);
     \draw[-] (2) to node {} (5);
     \draw[-] (2) to node {} (6);
     \draw[-] (3) to node {} (7);
     \draw[-] (4) to node {} (7);
     \draw[-] (5) to node {} (7);
     \draw[-] (6) to node {} (7);
\end{tikzpicture}
\end{minipage}
\hspace{3.5cm}
\begin{minipage}{.15\textwidth}
\begin{tikzpicture}[node distance = 1.5cm, auto]
     \node (1) {$L(E[2])$};
     \node (2)[right=of 1]   {$L(E[3])$};
     \node (3)[below=of 1]  {$L(\sqrt{-21})=L(\sqrt{21})$};
     \node (5)[right of=4]  {$L(\sqrt{-7})$};
     \node (6)[right of=5]  {$L(\sqrt{-3})$};
      \node (7)[below of=3, right of = 3]  {$L$};
     \draw[-] (1) to node {} (3);
     \draw[-] (2) to node {} (3);
     \draw[-] (2) to node {} (5);
     \draw[-] (2) to node {} (6);
     \draw[-] (3) to node {} (7);
     \draw[-] (5) to node {} (7);
     \draw[-] (6) to node {} (7);
\end{tikzpicture}
\end{minipage}
\end{example}

\begin{remark}\label{Z2Q}\nf
If $E$ is defined over $\Q$ then Theorem \ref{main} implies that 
$$
\Z/2\Z\in \Ent_{p,q}(E/\Q)
$$
for all distinct odd primes $p$ and $q$. This is due to the fact that $\Q(\zeta_p)\subseteq \Q(E[p])$ (by the Weil pairing), so $2$ divides $d(E,p,q)$ whenever $p$ and $q$ are both odd. When $p=2$, since the density of  elliptic curves with nonsquare discriminant is 1 (those with square discriminant have $j$-invariant $1728+t^2$), one still expects $\Z/2\Z\in \Ent_{2,q}(E/\Q) $ for `most' elliptic curves. 
\end{remark}

If both the mod $p$ and $q$ representations of $E/K$ are surjective (true for all but finitely many $(p,q)$ by \cite{Ser72}) then 
$$
d(E,p,q)=\gcd\big((p^2-1)(p^2-p),(q^2-1)(q^2-q)\big).
$$
It is straightforward to check that this integer is always divisible by 6 (for any choice of $p$ and $q$) and that it is divisible by $p$ if $q\equiv \pm 1\Mod p$. Theorem \ref{main} and divisibility \eqref{div1} immediately imply the following. 

\begin{corollary}\label{classification} If  $\rho_{E,p}$ and $\rho_{E,q}$ are surjective then
\begin{enumerate}
\item $\Z/2\Z,\Z/3\Z \in \Ent_{p,q}(E/K)$, and
\item if $q\equiv \pm 1\Mod p$ then $\Z/p\Z \in \Ent_{p,q}(E/K)$.
\end{enumerate}
\end{corollary}

We take this idea a step further in Theorem \ref{2qclassification} where we show that when $p=2$,
$$
\Ent_{2,q}(E/\Q)\subseteq \{1,\Z/2\Z,\Z/3\Z, S_3\},
$$
 for all odd primes $q$, with equality if $\im \rho_{E,2q}$ is maximal. 

The size of $d(E,p,q)$ can be drastically smaller when either $p$ or $q$ is exceptional, or when $E$ has complex multiplication, and in these cases the divisibility \eqref{div1} forces $\Ent_{p,q}(E/K)$ to also shrink in size. We remark that, for any given elliptic curve $E/\Q$, the number $d(E,p,q)$
is explicitly computable using \cite{Zyw15}. 

The proof of Theorem \ref{main} is constructive in the sense that it describes an explicit field $L$ (as the fixed field of a certain subgroup of $\Gal(K(E[pq]/K)$) which yields the described cyclic entanglement. It does not, however, describe \emph{all such} fields. Indeed, if there exists some $L_0/K$ such that $T_{p,q}(E/L_0)=T$ for a fixed group $T$, we suspect that the set 
\begin{equation}\label{Fset}
\{L/K \mid T_{p,q}(E/L)=T\}
\end{equation}
is infinite. We prove this for $T=\Z/\ell\Z$, where $\ell$ is as in Theorem \ref{main} (see Theorem \ref{infmanyfields}).  Note that the set \eqref{Fset}  has an obvious partial order under field inclusion. It would be interesting to obtain a description of the \emph{minimal} fields $L$ which achieve a desired entanglement. We suspect that such minimal fields always occur as subfields of $K(E[pq])$.

\subsection{Related questions and results}\label{relatedQs} There has been much work in recent years on classifying all elliptic curves defined over a \emph{fixed} base field $K$ (often $K=\Q$) having some prescribed entanglement properties. See, for example, \cite{BJ16}, \cite{Mor19}, \cite{DM20}, \cite{DLM21}, and  \cite{JM22}, where variations of the following question are considered. 

\begin{question} \it For which tuples $(p,q,T)$ do there exist infinitely many (non $\bar K$-isomorphic) elliptic curves $E/K$ having $(p,q)$-entanglement of type $T$? 
\end{question}

This type of question is amenable to study using the theory of modular curves, where (roughly speaking) one constructs a modular curve whose $K$-rational points correspond to elliptic curves with the desired entanglement properties. For example, by considering certain genus 0 and genus 1 (positive rank) modular curves, it is shown in \cite{DM20} that the only \emph{unexplained} (see Definition 4.5 of \emph{op. cit.}) entanglement types $T$ which can occur for infinitely many non-$\bar \Q$-isomorphic elliptic curves are $ \Z/2\Z$, $\Z/3\Z$, and $S_3$, and that they can occur only at the pairs $(p,q)$ in the sets $\{(2,3),(2,5),(2,7),(2,13),(3,5)\}$, $\{(2,3),(2,5),(2,7)\}$, and $\{(2,3)\}$, respectively. %A similar classification of non-abelian entanglements is carried out in \cite{JM22}, where it is shown that (outside of possible sporadic entanglements coming from points on higher genus curves) the only non-abelian entanglement type for an infinitely one-parameter family of elliptic curves is $S_3$. 

We take a different approach, by instead beginning with a fixed tuple $(p,q,E/K)$, and considering what $(p,q)$-entanglements are produced upon base extensions by $L/K$. Other related questions are as follows: 

\begin{enumerate}

\item[(A)] 
Is there an algorithm that computes $\Ent_{p,q}(E/K)$ directly from the group structure of $\im\rho_{E,p}$, $\im\rho_{E,q}$, and $\im\rho_{E,pq}$? In other words, if one knows the images mod $p$ and $q$, and how the $(p,q)$-entanglement over $K$ fits into them, is it possible to determine all $(p,q)$-entanglement types over every extension of $K$? We discuss this question when $p=2$ in \S\ref{classificationsection}. 

\item[(B)] For fixed $(K,p,q)$, does there always exist an elliptic curve $E/K$ having nontrivial $(p,q)$-entanglement over $K$? When $K=\Q$ and $p=2$, the answer to this question is `yes': one can force \emph{explained} (see \cite{DM20}) entanglement using the Weil pairing. Specifically, any elliptic curve $E/\Q$ with $j$-invariant 
 $
 1728\pm qt^2,
 $
 where $\pm$ is the sign of $(-1)^{\frac{q-1}{2}}$ and $t\in \Q^\times$, will have $(2,q)$-entanglement: the $j$-invariant forces $\Delta_{E}\equiv  \pm q\Mod \Q^{\times 2}$, thus  $\Q(\sqrt{\Delta_E})=\Q(\sqrt{\pm q})$ is contained in both $\Q(E[2])$ and $\Q(\zeta_q)\subseteq \Q(E[q])$.

\item[(C)] 
Suppose that $E$ has no $(p,q)$-entanglement over its minimal basefield $K=\Q(j(E))$. Is the set  $\Ent_{p,q}(E/K)$ always nontrivial? That is, can one always base extend by some $L/K$ so that $E$ gains $(p,q)$-entanglement? When $p$ and $q$ are odd and $E$ is defined over $\Q$, the answer to this question is `yes' (see Remark \ref{Z2Q}). See also Remark \ref{Cremark}.
\end{enumerate}

\subsection{Acknowledgments}
 The authors thank Harris B. Daniels for many inspiring conversations and crucial suggestions throughout the preparation of this article. We are also grateful to Jeffrey Hatley for several helpful comments and corrections.

%%%%%%%%
\section{Entanglement Over Number Fields}

Fix an algebraic extension $K/\Q$, an elliptic curve $E/K$, and distinct primes $p$ and $q$. Define
\begin{align*}
K_{p,q}(E)&=K(E[p])\cap K(E[q]),\\
T_{p,q}(E/K)&= \Gal\big(K_{p,q}(E)/K\big),\\
\Ent_{p,q}(E/K)&=\{T_{p,q}(E/L) \mid \text{$L/K$} \}.
\end{align*}
We often suppress the elliptic curve $E$ from the notation for the $(p,q)$-entanglement field  and simply write $K_{p,q}$. Since the compositum 
$K\Q(E[p])=K(E[p])$ we note that $[K(E[p]):K]=[\Q(E[p]):\Q(E[p])\cap K]$ divides the integer $[\Q(E[p]):\Q]$, hence the group $T_{p,q}(E/K)$ is always finite. 

%%%%%%
\begin{definition}  \nf 
We say that $E$ has \emph{$(p,q)$-entanglement} over an extension $L/K$ if $T_{p,q}(E/L)$ is nontrivial. The \emph{entanglement type} of $E$ over $L$ is the isomorphism class of $T_{p,q}(E/L)$. 
\end{definition}

%%%%
\begin{proposition}\label{sizebound} For any algebraic extension $L/K$ we have 
$$
\#T_{p,q}(E/L)\,\, \big| \,\,  \gcd\big([K(E[p]):K],[K(E[q]):K]\big).
$$
In particular, $\displaystyle\sup_{L/K}\# T_{p,q}(E/L)<\infty$ and the set $\Ent_{p,q}(E/K)$ is finite. 
\end{proposition}
\begin{proof} Observe that
\begin{align*}
\#T_{p,q}(E/L)&=[L_{p,q}:L]\\
&\big| \,\, \gcd\big([L(E[p]):L],[L(E[q]):L]\big)\\
&= \gcd\big([K(E[p]):L\cap K(E[p])],[K(E[q]):L\cap K(E[q])]\big)\\
&\,\, \big| \,\, \gcd\big([K(E[p]):K],[K(E[q]):K]\big).
\end{align*}
\end{proof}

\subsection{Group-theoretic formulation} 
Fix a basis for $E[pq]$ and let $\rho_{E,pq} :G_K\rightarrow \GL_2(\Z/pq\Z)$ denote the Galois representation corresponding to the action of $G_K$ on $E[pq]$. Following \cite{DM20}, the group $\Gg :=\im\rho_{E,pq}$ is said to \emph{represent} a $(p,q)$-entanglement if 
$$
\langle \ker\pi_p\cap \Gg,\ker\pi_q\cap \Gg\rangle \subsetneq \Gg,
$$
where $\pi_p$ and $\pi_q$ are the natural reduction mod $p$ and $q$ maps on $\GL_2(\Z/pq\Z)$. In particular,  $E$ has $(p,q)$-entanglement if and only if $\Gg$ represents a $(p,q)$-entanglement. From this perspective, the entanglement type of $E$ over $K$ is
$$
T_{p,q}(E/K)\cong \Gg/\langle \ker\pi_p\cap \Gg,\ker\pi_q\cap \Gg\rangle. 
$$ 

%%%%%%%%%%%%
\subsection{Stable and decreasing entanglement}

A simple way to decrease entanglement size is to base-extend by a subfield $L$ of the $(p,q)$-entanglement field $K_{p,q}$, in which case  one can identify the entanglement over $L$ with a subgroup of the entanglement over $K$. 

\begin{proposition}\label{prop-subshrink} If $K\subseteq L\subseteq K_{p,q}$ then 
$T_{p,q}(E/L)\cong \Gal(K_{p,q}/L)\subseteq T_{p,q}(E/K)$.
\end{proposition}
\begin{proof} 
This follows from the fact that if $L\subseteq K_{p,q}$ then $L(E[p])=K(E[p])$ and $L(E[q])=K(E[q])$.
\end{proof}

\begin{example}\label{ex-2in3}\nf Let $E/\Q$ be any elliptic curve with $j$-invariant of the form 
 $$
2^{10} 3^3 t^3(1-4t^3)
 $$
 for some $t\in \Q\backslash\{0,\frac{1}{2}\}$.  (For instance, taking $t=\frac{2}{3}$ gives the elliptic curve \href{https://www.lmfdb.org/EllipticCurve/Q/300a1/}{\texttt{300a1}}.) Then it follows from \cite[Remark 1.5]{BJ16} and \cite[\S 8.1]{DM20}, respectively, that 
 \begin{itemize}
\item $E$ has $(2,3)$-entanglement of type $S_3$, and 
\item $\Q(E[2])\subseteq \Q(E[3])$.
\end{itemize}
The 2-torsion field contains the quadratic field $\Q(\sqrt{\Delta_E})$ and three distinct cubic subfields, which we denote by $L_1,L_2$, and $L_3$. It follows that
$$
T_{2,3}(E/L)=\begin{cases} S_3&\text{if $L=\Q$} \\
\Z/3\Z & \text{if $L=\Q(\sqrt{\Delta_E})$}\\
\Z/2\Z& \text{if $L=L_i$.}
\end{cases}
$$
\end{example}

The relation between $T_{p,q}(E/L)$ and $T_{p,q}(E/K)$ 
for arbitrary $L/K$ is illustrated in the following diagram.
\begin{figure}[H]
\begin{center}    
\begin{tikzpicture}[node distance = 1.5cm, auto]
     \node (1) {$L_{p,q}$};
     \node (2)[below of=1]  {$K_{p,q}L$};
     \node (3) [below of=2, left of =2] {$K_{p,q}$};
     \node (4) [below of=2, right of=2] {$ L$};
     \node (5) [below of=3, right of=3] {$K_{p,q}\cap L$};
     \node (6) [below of=5] {$K$};
     \draw[-] (1) to node {} (2);
     \draw[-] (2) to node {} (3);
     \draw[-] (2) to node {} (4);
     \draw[-] (3) to node {} (5);
      \draw[-] (4) to node {} (5);
       \draw[-] (5) to node {} (6);
        \draw[-] (3) to node[xshift=-2cm,yshift=-5mm] {$T_{p,q}(E/K)$} (6);
        \draw[-] (1) to node[xshift=.0cm,yshift=0mm] {$T_{p,q}(E/L)$}(4);
\end{tikzpicture}
\end{center}
\caption{}
\label{LKdiagram}
\end{figure}

Notice that if $LK_{p,q}=L_{p,q}$ then one can again identify the entanglement over $L$ with a subgroup of the entanglement over $K$. However, this need not be the case: for instance, in Example \ref{quadex}  one has $LK_{p,q}=\Q(\sqrt{-1})$ but $L_{p,q}=\Q(\sqrt{-1},\sqrt{21})$. The following lemma clarifies this relationship. 

%%%%%
\begin{lemma}\label{lem-comp} Let $L$, $K_1,$ and $K_2$ be finite extensions of $K$. Then $L(K_1\cap K_2)=LK_1\cap LK_2$ if and only if $LK_1$ and $LK_2$ are linearly disjoint over $L(K_1\cap K_2)$. 
\end{lemma}
\begin{proof} Since $L(K_1\cap K_2)$ is the smallest field containing both $L$ and $K_1\cap K_2$, we know that $L(K_1\cap K_2)\subseteq LK_1\cap LK_2$. Write
\begin{align*}
a_1&=[LK_{1}:L(K_1\cap K_2)]\\
a_2&=[LK_{2}:L(K_1\cap K_2)]\\
b &=[LK_1\cap LK_2:L(K_1\cap K_2)].
\end{align*}
Then $LK_1$ and $LK_2$ are linearly disjoint over $L(K_1\cap K_2)$ if and only if $[LK_1K_2:L(K_1\cap K_2)]=a_1a_2$. It therefore suffices to show that $b=1$ if and only if $[LK_1K_2:L(K_1\cap K_2)]=a_1a_2$, which follows from the diagram below, where the dashed lines indicate extensions of the same degree. 
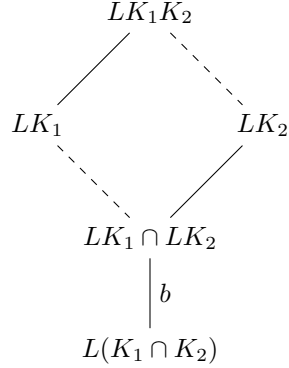
\begin{center}    
\begin{tikzpicture}[node distance = 1.5cm, auto]
     \node (1) {$LK_1K_2$};
     \node (2)[below of=1, left of =1]  {$LK_1$};
     \node (3) [below of=1, right of =1] {$LK_2$};
     \node (4) [below of=2, right of =2] {$ LK_1\cap LK_2$};
     \node (5) [below of=4] {$L(K_1\cap K_2)$};
     \draw[-] (1) to node[xshift=-5mm,yshift=4mm] {} (2);
     \draw[dashed] (1) to node {} (3);
     \draw[dashed] (2) to node[xshift=-5mm,yshift=-4mm] {} (4);
     \draw[-] (3) to node {} (4);
      \draw[-] (4) to node {$b$} (5);
\end{tikzpicture}
\end{center}
\end{proof} 

%%%%%%
\begin{proposition}\label{shrinkstable} Let $L/K$ be a finite extension. If $L(E[p])$ and $L(E[q])$ are linearly disjoint over $LK_{p,q}$ then 
$$
T_{p,q}(E/L)\cong\Gal(K_{p,q}/L\cap K_{p,q}) 
$$
\end{proposition}
\begin{proof} From Lemma \ref{lem-comp}, we know that $K_{p,q}L=L_{p,q}$. The result now follows from the field diagram in Figure \ref{LKdiagram}.
\end{proof}

\begin{example} \nf Let $E/\Q$ be as in Example \ref{ex-2in3} above and let $L/\Q$ be any quadratic field not equal to $\Q(\sqrt{\Delta_E})$. Since $L(E[2])\subseteq L(E[3])$, clearly $L(E[2])$ and $L(E[3])$ are linearly disjoint over (their common intersection) $L(E[2])=L(\Q(E[2])\cap\Q(E[3]))$. Since $L\cap \Q(E[2])\cap\Q(E[3])=\Q$, it follows from Proposition \ref{shrinkstable} that the $(2,3)$-entanglement type does not change as we base extend from $\Q$ to $L$. More succinctly, combined with Example \ref{ex-2in3}, we find that as $D$ ranges over square-free integers,
$$
T_{2,3}(E/\Q(\sqrt{D}))=\begin{cases} \Z/3\Z&\text{if $D=\Delta_E$,} \\
S_3 & \text{otherwise.}
\end{cases}
$$
\end{example}

\subsection{Entanglement along subfields of $K(E[pq])$.}
The following proposition allows one to determine the size of the entanglement along subfields of $K(E[pq])$.

\begin{proposition}\label{LKlemma} If $L/K$ is a subfield of $K(E[pq])$ then
$$
\#T_{p,q}(E/L)= \#T_{p,q}(E/K) \frac{[L:K]}{[K(E[p])\cap L:K][K(E[q])\cap L:K]}.
$$
In particular, if $L$ trivially intersects $K(E[p])$ and $K(E[q])$ then
$$
\#T_{p,q}(E/L)=\#T_{p,q}(E/K)[L:K].
$$
\end{proposition}

\begin{proof} Since $K(E[pq])=L(E[pq])$ we have
\begin{equation}\label{Kpq1}
[K(E[pq]):L]=\frac{[L(E[p]):L] [L(E[q]):L]}{[L_{p,q}:L]}.
\end{equation}
 We can also write 
\begin{equation}\label{Kpq2}
[K(E[pq]):L]=\frac{[K(E[pq]):K]}{[L:K]}=\frac{[K(E[p]):K] [K(E[q]):K]}{[L:K][K_{p,q}:K]}.
\end{equation}
Combining \eqref{Kpq1} and \eqref{Kpq2} we have
\begin{align*}
[L_{p,q}:L]&=[L:K][K_{p,q}:K]\frac{[L(E[p]):L]}{[K(E[p]):K]}\frac{[L(E[q]):L]}{[K(E[q]):K]}.
\end{align*}
The result now follows from the fact that $[L(E[p]):L]= [K(E[p]):L\cap K(E[p])]$, and similarly for $q$. 
\end{proof}

\begin{remark}\label{Cremark} \nf Proposition \ref{LKlemma} provides a partial answer to Question \S\ref{relatedQs}(C): if $E$ has no entanglement over $K$ and there exists a nontrivial subfield $K\subseteq L\subseteq K(E[pq])$ that is disjoint (over $K$) from both the $p$ and $q$-torsion fields then 
$$
\#T_{p,q}(E/L)=[L:K]>1.
$$
\end{remark}

 \begin{corollary}\label{quadprop} Suppose that $E/\Q$ has no $(p,q)$-entanglement over $\Q$. If $\Q(\sqrt{D_1})$ and $\Q(\sqrt{D_2})$ are distinct quadratic subfields of $\Q(E[p])$ and $\Q(E[q])$, respectively, then 
 $$
 T_{p,q}(E/\Q(\sqrt{D_1D_2}))=\Z/2\Z.
 $$
 In particular, for all odd primes $p$ and $q$ we have 
$$
T_{p,q}(E/\Q(\sqrt{\pm pq}))=\Z/2\Z,
$$
where $\pm=\sgn(-1)^{\frac{p+q}{2}+1}$
\end{corollary}
\begin{proof} Since $E$ has no $(p,q)$-entanglement over $\Q$, the field $\Q(\sqrt{D_1D_2})$ must intersect both $\Q(E[p])$ and $\Q(E[q])$  trivially. (If $\Q(\sqrt{D_1D_2})$ were a subfield of $\Q(E[p])$, say, then $\Q(E[p])$ would also contain $\Q(\sqrt{D_2})$, a contradiction.) The first claim now follows from Proposition \ref{LKlemma}, and the second is due to the fact that $\Q(\sqrt{(-1)^{\frac{p-1}{2}}p})\subseteq \Q(E[p])$ and $\Q(\sqrt{(-1)^{\frac{q-1}{2}}q})\subseteq \Q(E[q])$. \end{proof}

\begin{remark}\label{quadexfull}\nf 
The existence of multiple distinct quadratic subfields of $p$ and $q$-torsion fields is directly related to whether or not the primes $p$ and $q$ are exceptional for $E$. In what follows, we write $\Q(\sqrt{p^*})$ for the unique quadratic subfield of $\Q(\zeta_p)$. There are three cases to consider:
\begin{enumerate}

\item The only quadratic subfields of $\Q(E[p])$ and $\Q(E[q])$ are those forced by the Weil pairing, i.e., the fields $\Q(\sqrt{p^*})$ and $\Q(\sqrt{q^*})$ contained in the cyclotomic parts of $\Q(E[p])$ and $\Q(E[q])$. In this case, base extending by $L=\Q(\sqrt{p^*q^*})$ forces a quadratic $(p,q)$-entanglement. (If $p=2$, replace $\Q(\sqrt{p^*})$ by $\Q(\sqrt{\Delta_E})$.)

\item Without loss of generality, only the $p$-torsion field contains a quadratic subfield $\Q(\sqrt{D})$ which is not forced by the Weil pairing. Then base extending to either $\Q(\sqrt{Dq^*})$ or $\Q(\sqrt{p^*q^*})$ will force quadratic $(p,q)$-entanglement. 

\item  Both $\Q(E[p])$  and $\Q(E[q])$ have distinct quadratic subfields, call them $\Q(\sqrt{D})$ and $\Q(\sqrt{D'})$, different from those forced by the Weil pairing. Then base-extending to any of the fields  $\Q(\sqrt{DD'})$, $\Q(\sqrt{ Dq^*})$, $\Q(\sqrt{D'p^*})$, or $\Q(\sqrt{p^*q^*})$ yields quadratic $(p,q)$-entanglement. 

\end{enumerate}
Since $\GL_2(\Z/p\Z)$ and $\GL_2(\Z/q\Z)$ have unique index 2 subgroups, the latter two cases can only occur when only $p$ (case (2)), or both $p$ and $q$ (case (3)), are \emph{exceptional} primes. (Recall that a prime is called exceptional for $E$ if the residual Galois representation at that prime is not surjective.) See Example \ref{quadex} for a specific curve illustrating case (2) above.
\end{remark}

 \begin{corollary} Suppose that $E/\Q$ has no $(p,q)$-entanglement over $\Q$. Then
$$
T_{p,q}(E/\Q(\zeta_{pq})^+)=\Z/2\Z
$$
where $\Q(\zeta_{pq})^+$  is the maximal real subfield of $\Q(\zeta_{pq})$.
\end{corollary}
\begin{proof} Since $E$ has no $(p,q)$-entanglement over $\Q$, Lemma \ref{maxreallem}  implies
%(with $F=\Q(E[\text{$p$ or $q$}])\cap \Q(\zeta_{pq})^+$) implies
\begin{align*}
\Q(E[p])\cap \Q(\zeta_{pq})^+&=\Q(\zeta_p)^+\\
\Q(E[q])\cap \Q(\zeta_{pq})^+&=\Q(\zeta_q)^+.
\end{align*}
It follows from Proposition \ref{LKlemma} that 
$$
\#T_{p,q}(E/\Q(\zeta_{pq})^+)=\frac{[\Q(\zeta_{pq})^+:\Q]}{[\Q(\zeta_{p})^+:\Q][\Q(\zeta_{q})^+:\Q]}=2.
$$
\end{proof}

\begin{lemma}\label{maxreallem} If $\Q(\zeta_p)^+\subseteq F\subseteq \Q(\zeta_{pq})^+$ and $F\cap \Q(\zeta_q)^+=\Q$ then $F=\Q(\zeta_p)^+$. 
\end{lemma}
\begin{proof}  Identify $\Gal(\Q(\zeta_{pq})^+/\Q(\zeta_p)^+)$ and $\Gal(\Q(\zeta_{pq})^+/\Q(\zeta_q)^+)$ with $\Z/\frac{q-1}{2}\Z$ and $\Z/\frac{p-1}{2}\Z$, respectively. If $F$ properly contains $\Q(\zeta_p)^+$ then $H=\Gal(\Q(\zeta_{pq})^+/F)$ is a proper subgroup of $\Z/\frac{q-1}{2}\Z$. But the fixed field of $\langle H, \Z/\frac{p-1}{2}\Z \rangle=H\times \Z/\frac{p-1}{2}\Z$ is $F\cap \Q(\zeta_q)^+=\Q$, and this forces $H=\Z/\frac{q-1}{2}\Z$, a contradiction. 
\end{proof}

%%%%%%
\section{Entangleable Fields} \label{entfieldsec}

 In this section we consider the notion of \emph{entangleable fields}.  Roughly speaking, a pair of fields are entangleable if they can be `glued together' by some auxilliary field. The relation between field entangleability and that of elliptic curves will be discussed in \S\ref{ellcurveent}. 
 
 Let $F_1$ and $F_2$ be finite extensions of a fixed algebraic extension $K/\Q$. 

\begin{definition} \nf We say that $F_1$ and $F_2$ are \emph{entangleable over $K$} if there exists some $L/K$ such that $LF_1=LF_2\supsetneq L$. The group $\Gal(LF_i/L)$ is called the \emph{entanglement type} of $F_1$ and $F_2$ with respect to $L$.
\end{definition}

Usually the field $K$ is clear from context, in which case we simply say that $F_1$ and $F_2$ are \emph{entangleable}. 

\begin{example}\label{quadex1} \nf Let $D_1$ and $D_2$ be distinct  squarefree integers. Then $\Q(\sqrt{D_1})$ and $\Q(\sqrt{D_2})$ are entangleable by $\Q(\sqrt{D_1D_2})$. This is easy to see directly from the definition, though we show it using Galois theory as well in Example \ref{groupquad} below.
\end{example}

 \begin{remark}\label{selfent}\nf  In the case where $F:=F_1=F_2$, we see that $F$ is \emph{self-entangleable} in the sense that $LF_1=LF_2=F$ has positive degree over $L$ for any proper subextension  $K\subseteq L\subsetneq F$.
 \end{remark}

We focus the remainder of this section on determining various sufficient conditions that will allow us to conclude when a pair of fields are entangleable. We begin by observing that not all field pairs are entangleable.

%%%
\begin{proposition}\label{coprime} If $[F_1:K]$ and $[F_2:K]$ are coprime then $F_1$ and $F_2$ are not entangleable. 
\end{proposition}
\begin{proof} Suppose they are entangleable by some $L$. Then the integer $d=[LF_1:L]$ is larger than 1 by assumption, but since we can also write $d=[F_1:F_1\cap L]=[F_2:F_2\cap L]$, we see that $d$ is a positive divisor of  $[F_1:K]$ and $[F_2:K]$, a contradiction. \,\,
\end{proof}

The coprimality hypothesis in Proposition \ref{coprime} is essentially the only obstruction to fields being entangleable (see Proposition \ref{groupcomp} below). In order to see this, we recast field entangleability in group-theoretic terms. Let $F/K$ be any finite Galois extension containing both $F_1$ and $F_2$ and write $G=\Gal(F/K)$, $G_1= \Gal(F/F_1)$, and  $G_2= \Gal(F/F_2)$.

\begin{lemma}\label{group-ent} $F_1$ and $F_2$ are entangleable by a subfield $L\subseteq F$ if and only if there exists a subgroup $H\leq G$ such that $H\cap G_1=H\cap G_2\subsetneq H$. If this is the case, then $L$ and $H$ are related by the correspondence
$$
L=F^H\longleftrightarrow H=\Gal(F/L).
$$
Furthermore, if $L/K$ or $F_i/K$ is Galois then $\Gal(LF_i/L)\cong H/(H\cap G_i)$.
\end{lemma}

\begin{proof} This follows trivially from the Galois correspondence and the fact that the compositum of two fields corresponds to the intersection of the corresponding groups. 
\end{proof}

In the next few examples, we illustrate the utility of Lemma \ref{group-ent}.  

\begin{example}\label{groupquad} \nf Let $D_1$ and $D_2$ be squarefree integers. Let $\sigma$ and $\tau$ be generators of $G=\Gal(\Q(\sqrt{D_1},\sqrt{D_2})/\Q)$ that restrict to the nontrivial elements of $\Gal(\Q(\sqrt{D_1})/\Q)$ and $\Gal(\Q(\sqrt{D_2})/\Q)$, respectively. Then the subgroup $H\leq G$ generated by $\sigma\tau$ trivially intersects $G_i=\Gal(\Q(\sqrt{D_1},\sqrt{D_2})/\Q(\sqrt{D_i}))$, hence $\Q(\sqrt{D_1})$ and $\Q(\sqrt{D_2})$ are entangleable by $\Q(\sqrt{D_1},\sqrt{D_2})^H=\Q(\sqrt{D_1D_2})$.
\end{example}

\begin{example}\label{cyclotomic}\nf Fix positive integers $n$ and $m$, neither of which are equal to $2$. The fields $\Q(\zeta_n)$ and $\Q(\zeta_m)$ are entangleable by $\Q(\zeta_{mn})^+:=\Q(\zeta_{nm}+\zeta_{nm}^{-1})$, the maximal real subfield of $\Q(\zeta_{nm})$. (If $n=m$ this is obvious -- see Remark \ref{selfent}.) To see this, note that the group $G=\Gal(\Q(\zeta_{nm})/\Q)$ is canonically isomorphic to $(\Z/(\phi(nm)\Z))^\times$, wherein $G_1=\Gal(\Q(\zeta_{nm})/\Q(\zeta_n))$ and $G_2=\Gal(\Q(\zeta_{nm})/\Q(\zeta_m))$ are identified with integers congruent to 1 mod $n$ and $m$, respectively. One can now check that the subgroup $H\leq G$ generated by $-1$ (complex conjugation) satisfies the hypotheses of Lemma \ref{group-ent} (again one has $G_i\cap H=1$). 
\end{example}

\begin{example}\label{S3}\nf Let $F_1$ and $F_2$ be any two distinct $S_3$-extensions of $\Q$ and set $F=F_1F_2$. Suppose that $F_1\cap F_2$ is quadratic over $\Q$. Then
$$
G=\Gal(F/\Q)\cong S_3\times_{\Z/2\Z}S_3.
$$
Letting $G_i=\Gal(F/F_i)$, we compute in Magma that  there are exactly five subgroups $H_1, \dots, H_5\leq G$, satisfying the hypotheses of Lemma \ref{group-ent}. They all intersect $G_1$ and $G_2$ trivially and their isomorphism classes are:
\begin{itemize}
\item $H_1=\Z/2\Z$
\item $H_2=H_2=\Z/3\Z$
\item $H_4=H_5=S_3$.
\end{itemize}
It follows that there are five subfields $L_i=F^{H_i}$, each of which entangle $F_1$ and $F_2$, and the entanglement types are:
\begin{itemize}
\item $\Gal(L_1F_1/L_1)=\Z/2\Z$
\item $\Gal(L_iF_1/L_i)=\Z/3\Z$ for $i=2,3$
\item $\Gal(L_iF_1/L_i)=S_3$ for $i=4,5$.
\end{itemize}
(Since $LF_1=LF_2$, each of the above also holds for $F_2$.) 
\end{example}

\subsection{Criteria for field entangleability}

\begin{proposition}\label{entdeg} 
Suppose that $F_1$ and $F_2$ are Galois over $K$ and that  $d:= \gcd([F_1:F_1\cap F_2],[F_2:F_1\cap F_2])>1$. For each prime prime $\ell\mid d$ there is a subfield $L\subseteq F_1F_2$ such that $F_1$ and $F_2$ are entangleable by $L$ and $[LF_i:L]=\ell$.
%Suppose that $F_1$ and $F_2$ are Galois over $K$ and that  $d:= \gcd([F_1:F_1\cap F_2],[F_2:F_1\cap F_2])>1$. For each prime prime $\ell\mid d$ there is a subfield $L\subseteq F_1F_2$ such that $F_1$ and $F_2$ are entangleable by $L$ and $[LF_i:L]=\ell$.
\end{proposition}
\begin{proof} Recall that $G=\Gal(F_1F_2/K)$ is isomorphic to the fibre product
$$
\Gal(F_1/K) \times_{\Gal(F_1\cap F_2/K)} \Gal(F_2/K)=\{(g_1,g_2)\mid g_1|_{F_1\cap F_2}=g_2|_{F_1\cap F_2}\},
$$
under the mapping $g\mapsto (g|_{F_1},g|_{F_2})$. Let $\ell$ be a prime divisor of $d$. By Cauchy's theorem, there are elements $g_i\in \Gal(F_i/F_1\cap F_2)$ of order $\ell$. Let $H$ be the preimage in $G$ of the subgroup $\langle (g_1,g_2)\rangle$ of the fibre product above. (Note that $(g_1,g_2)$ is indeed an element of the fibre product.) Then $H$ has order $\ell$ and $H\cap G_i$ is trivial. To see this, let $g\in G$ be the element mapping to $(g_1,g_2)$  (thus $H=\langle g\rangle$) and observe that if $g\in G_i=\Gal(F_1F_2/F_i)$ then $g|_{F_i}=1=g_i$, a contradiction. It follows from Lemma \ref{group-ent} that $F_1$ and $F_2$ are entangleable by $L=F^H$.
\end{proof}

Notice that in Examples \ref{groupquad} and \ref{cyclotomic}, entanglement was always achieved via a subfield of $F_1F_2$. This turns out to be true in general, as we now show. 

\begin{theorem}\label{ent-comp} Suppose that $F_1$ and $F_2$ are entangleable by a field $L/K$. If either

\begin{enumerate}

\item[(i)] $F_1$ and $F_2$ are both Galois over $K$, or 

\item[(ii)]  $L$ is Galois over $K$,

\end{enumerate}

\noindent then $F_1$ and $F_2$ are entangleable by
$L_0:=L\cap F_1F_2$ and 
$$
\Gal(L_0F_i/L_0)\cong\Gal(LF_i/L)
$$
for $i=1,2$.
\end{theorem}

\begin{proof} Write $F$ for  the Galois closure of $LF_1F_2$. Let $G=\Gal(F/K)$, $G_1= \Gal(F/F_1)$,  $G_2= \Gal(F/F_2)$, and $H=\Gal(F/L)$.
Under the Galois correspondence, we have
\begin{align*}
LF_i&\longleftrightarrow H\cap G_i\\
L_0 &\longleftrightarrow \langle H,G_1\cap G_2\rangle\\
L_0F_i&\longleftrightarrow \langle H,G_1\cap G_2\rangle\cap G_i.
\end{align*} 
Since $F_1$ and $F_2$ are entangleable by $L$, we know that $H\cap G_1=H\cap G_2\subsetneq H$. If $F_1$ and $F_2$ are Galois over $K$, the subgroup $G_1\cap G_2$ (corresponding to $F_1F_2$) is normal in $G$. If $L$ is Galois over $K$, then $H$ is normal in $G$. The fact that $F_1$ and $F_2$ are entangleable by $L_0$ now follows from Proposition \ref{groupcomp} below.

The final statement follows from the diagram below, where the dashed lines indicate extensions with isomorphic Galois groups.
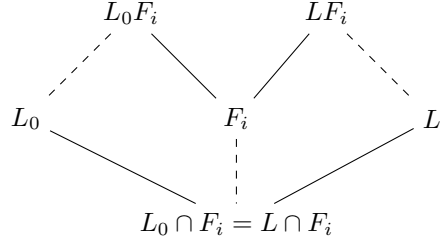
\begin{center}    
\begin{tikzpicture}[node distance = 1.4cm]
\node (1) {$L_0F_i$};
\node (2)[below of =1, left of =1]  {$L_0$};
\node (3)[below of =1, right of =1]  {$F_i$};
\node (4) [below of=3] {$L_0\cap F_i=L\cap F_i$};
\node (5) [xshift=1.2cm][right of=1] {$LF_i$};
\node (6) [below of=5, right of=5] {$L$};
\draw[dashed] (1) to node {} (2);
\draw[-] (1) to node {} (3);
\draw[dashed] (3) to node {} (4);
\draw[-] (2) to node {} (4);
\draw[-] (5) to node {} (3);
\draw[dashed] (5) to node {} (6);
\draw[-] (6) to node {} (4);
\end{tikzpicture}
\end{center}
\end{proof}

\begin{proposition}\label{groupcomp} Let $G$ be a group with subgroups $G_1,G_2$, and $H$. Suppose that either $H$ or $G_1\cap G_2$ is normal in $G$. Then $H\cap G_1=H\cap G_2\subsetneq H$ implies 
$$
\langle H,G_1\cap G_2\rangle\cap G_1=\langle H,G_1\cap G_2\rangle\cap G_2\subsetneq \langle H,G_1\cap G_2\rangle. 
$$
\end{proposition}
\begin{proof} Since $H$ or  $G_1\cap G_2$ is normal in $G$, we know that $H(G_1\cap G_2)$ is a subgroup of $G$, hence $\langle H,G_1\cap G_2\rangle =H(G_1\cap G_2)$. The strict inclusion in the statement is now obvious. It therefore suffices to prove $\langle H,G_1\cap G_2\rangle\cap G_1\subseteq \langle H,G_1\cap G_2\rangle\cap G_2$, since the other inclusion will then follow by symmetry. Let $x\in H(G_1\cap G_2)\cap G_1$. We must show $x\in G_2$. Since $x\in H(G_1\cap G_2)$ we can write $x=hg$ for some $h \in H$ and $g\in G_1\cap G_2$. Then $h=xg^{-1}\in G_1$ since both $x\in G_1$ and $g^{-1}\in G_1$. Thus $h\in H\cap G_1=H\cap G_2$, so in particular $h\in G_2$. It now follows that $x=hg\in G_2$ since $h\in G_2$ and $g\in G_2$. \end{proof}

Theorem \ref{ent-comp} immediately implies the following.

\begin{corollary}\label{coro-classifyent} Suppose that $F_1$ and $F_2$ are Galois over $K$. Then $F_1$ and $F_2$ are entangleable if and only if $F_1$ and $F_2$ are entangleable by a subfield of $F_1F_2$. \end{corollary}

If $F_1$ and $F_2$ are entangleable by some $L$, the following proposition shows that one can often use $L$ to construct other fields which also entangle $F_1$ and $F_2$.

%%%
\begin{proposition}\label{compent} Suppose that $F_1$ and $F_2$ are entangleable by $L/K$. If $L'/K$ satisfies $LL'\cap F_i=L\cap F_i$ for either $i=1$ or $2$, then $F_1$ and $F_2$ are entangleable by $LL'$ and 
$$
[LL'F_i:LL']=[LF_i:L].
$$
\end{proposition}
\begin{proof} The result follows from the fact
$$
(LL')F_1=L'(LF_1)=L'(LF_2)=(LL')F_2
$$
and the diagram below, where dashed lines indicate extensions of the same degree. 
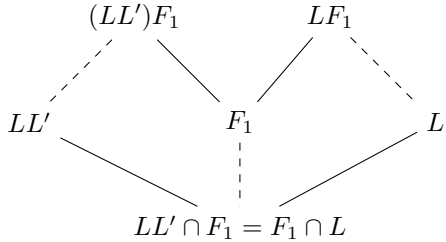
\begin{center}    
\begin{tikzpicture}[node distance = 1.4cm]
\node (1) {$(LL')F_1$};
\node (2)[below of =1, left of =1]  {$LL'$};
\node (3)[below of =1, right of =1]  {$F_1$};
\node (4) [below of=3] {$LL'\cap F_1=F_1\cap L$};
\node (5) [xshift=1.2cm][right of=1] {$LF_1$};
\node (6) [below of=5, right of=5] {$L$};
\draw[dashed] (1) to node {} (2);
\draw[-] (1) to node {} (3);
\draw[dashed] (3) to node {} (4);
\draw[-] (2) to node {} (4);
\draw[-] (5) to node {} (3);
\draw[dashed] (5) to node {} (6);
\draw[-] (6) to node {} (4);
\end{tikzpicture}
\end{center}
\end{proof}

\begin{lemma}\label{inffield} There exist infinitely many fields $L'/K$ such that $LL'\cap F_1=L\cap F_1$.
\end{lemma}
\begin{proof} We prove there are infinitely many quadratic extensions $L'/K$ satisfying the claim. Since $L/K$ and $F_1/K$ have finite degree, the sets
$$
S_1 = \{D\in \Oo_K\mid K\subsetneq K(\sqrt{D})\subseteq F_1\}\quad \text{and}\quad S_2 = \{D\in \Oo_K\mid K\subsetneq K(\sqrt{D})\subseteq L\}
$$
must be finite. Let $D$ be any squarefree element of $\Oo_K$ that is not in $S_1\cup S_2$ and set $L'=L(\sqrt{D})$. We claim that $LL'\cap F_1=L\cap F_1$. One inclusion is obvious. For the other, let $x\in L(\sqrt{D})\cap F_1$. Since $x\in L(\sqrt{D})$ we can write $x=\ell_1+\sqrt{D}\ell_2$ for some $\ell_i\in L$. Since $x\in F_1$ and $F_1$ does not contain $\sqrt{D}$, we must have $\ell_2=0$, whence $x=\ell_1\in L\cap F_1$. 
\end{proof}

Proposition \ref{compent} and Lemma \ref{inffield} immediately imply the following corollary. 

\begin{corollary} If $F_1/K$ and $F_2/K$ are entangleable by some $L/K$ then they are entangleable by infinitely many $L/K$. 
\end{corollary}

%%%%%%%%%%%%%%%%%%%%%% 
 \section{Applications of Field Entanglement to Elliptic Curves}\label{ellcurveent}

One can use the idea of field entanglements to study the size of the $(p,q)$-entanglement of $E$, as illustrated in the following proposition. 

\begin{proposition}\label{entprop1} Let $F_1$ and $F_2$ be (not necessarily distinct) subfields of $K(E[p])$ and $K(E[q])$, respectively. Suppose that $F_1$ and $F_2$ are entangleable by some $L/K$. Then, 
\begin{enumerate}
\item $[LF_1:L]$ divides $\#T_{p,q}(E/L)$, and

\item if $F_1=K(E[p])$ then
$$
T_{p,q}(E/L)\cong\Gal(LF_1/L)\cong\Gal(K(E[p])/L\cap K(E[p]))\cong\Gal(F_2/L\cap F_2).
$$
\end{enumerate} 
\end{proposition}

\begin{proof} (1) Since $F_1\subseteq K(E[p])$ we know that $LF_1\subseteq LK(E[p])=L(E[p])$. Similarly, $LF_2\subseteq L(E[q])$. Thus, the field $LF_1=LF_2$ is contained in both $L(E[p])$ and $L(E[q])$, i.e., $LF_1\subseteq L_{p,q}$, whence $[LF_1:L]$ divides $[L_{p,q}:L]$.

(2) Suppose $F_1=K(E[p])$. Then
$$
L(E[p])=LF_1=L_{p,q}.
$$ 
The first equality is obvious, and the second follows from the inclusion $LF_1\subseteq L_{p,q}$ (shown in part (1)) together with the fact that if $x\in L_{p,q}$ then $x\in L(E[p])=LF_1$. The result now follows from the fact that $LF_1=LF_2$ and
$$
\Gal(LF_i/L)\cong\Gal(F_i/L\cap F_i).
$$
\end{proof}

Notice in particular that if $T_{p,q}(E/K)=1$ and if $K(E[p])$ and $K(E[q])$ contain a pair of subfields that are entangleable by some $L/K$, then Proposition \ref{entprop1} guarantees $T_{p,q}(E/L)\neq 1$. In other words, one can \emph{force} $(p,q)$-entanglements on $E$ by entangling various subfields of its $p$ and $q$-torsion fields. (See Example \ref{quadex} for a concrete illustration of entangling quadratic subfields).

\begin{example}\label{S3S3ex} \nf Let $E/\Q$ be the non-CM elliptic curve with Cremona label \href{https://www.lmfdb.org/EllipticCurve/Q/1323l1/}{\texttt{1323l1}}. 
Neither 2 nor 3 are exceptional for $E$, so $\Gal(\Q(E[2])/\Q)=S_3$ and $\Gal(\Q(E[3])/\Q)=\GL_2(\Z/3\Z)$. Observe that $E$ has quadratic $(2,3)$-entanglement over $\Q$ since $\Q(\sqrt{\Delta_E})=\Q(\sqrt{-3})$ and $\Q(E[2])$ is not contained in $\Q(E[3])$. Let $F_1=\Q(E[2])$ and let $F_2$ be the Galois closure of the (unique) cubic subfield of $\Q(E[3])$. Then both $F_1$ and $F_2$ are $S_3$-extensions of $\Q$ and $F_1\cap F_2=\Q(\sqrt{-3})$. By Example \ref{S3}, there are exactly five fields $L_1,\dots, L_5\subseteq F_1F_2$ which entangle $F_1$ and $F_2$, and it follows from Proposition \ref{entprop1}(2) and Example \ref{S3} that 
\begin{align*}
T_{2,3}(E/L)=\begin{cases} \Z/2\Z, &\text{if $L=\Q$ or  $L_1$}\\
1, &\text{if $L=\Q(\sqrt{-3})$}\\
\Z/3\Z, &\text{if $L=L_2$ or $L_3$}\\
S_3 &\text{if $L=L_4$ or $L_5$.}\\
\end{cases}
\end{align*}
(The same holds for any elliptic curve with $j$-invariant $1728-3t^2$, $t\in \Q^\times$, and which has maximal image mod $2$ and mod 3.) 
\end{example}

 \subsection{Cyclic entanglements}
 
\begin{theorem}\label{main1} Let $\ell$ be a prime such that $\ell\mid\gcd\big([K(E[p]):K],[K(E[q]):K]\big)$. Then there exists a field $K\subseteq L\subseteq K(E[pq])$ such that
$
T_{p,q}(E/L)= \Z/\ell\Z.
$
\end{theorem}

\begin{proof} 
First, suppose that $\ell\mid [K_{p,q}:K]$. Then there is a subfield $K\subseteq L\subseteq K_{p,q}$ such that $[K_{p,q}:L]=\ell$. It follows from Proposition \ref{prop-subshrink} that $T_{p,q}(E/L)=\Z/\ell\Z$. 

Now assume that $\ell\nmid [K_{p,q}:K]$, and set $F_1=K(E[p])$ and $F_2=K(E[q])$. Then  $\ell\mid \gcd([F_1:F_1\cap F_2],[F_2:F_1\cap F_2])$ and it follows from Proposition \ref{entdeg} that $F_1$ and $F_2$ are entangleable by a subfield $L\subseteq F_1F_2$ with $[LF_i:L]=\ell$. The result now follows from Proposition \ref{entprop1}.
 \end{proof}
 
 \begin{corollary} Suppose that 
 
 (i) $p>2$, or $p=2$ and $\Delta_E\notin K^{\times 2}$, and 
 
 (ii) $\sqrt{p^*},\sqrt{q^*}\notin K$. 

\noindent Then for any prime $q\geq 3$ there exists a field $K\subseteq L\subseteq K(E[pq])$ such that 
$$
T_{p,q}(E/L)=\Z/2\Z.
$$
\end{corollary}
\begin{proof} If $p$ and $q$ are both odd then $K(\sqrt{p^*})$ and $K(\sqrt{q^*})$ are nontrivial quadratic extensions of $K$ contained in $K(\zeta_p)\subseteq K(E[p])$ and $K(\zeta_q)\subseteq K(E[q])$, respectively. In particular, $\gcd\big([K(E[p]):K],[K(E[q]):K]\big)$ is divisible by 2 and the result follows from Theorem \ref{main1}. If $p=2$ and $\Delta_E\notin K^{\times 2}$ then $K(\sqrt{\Delta_E})$ is a quadratic extension of $K$ contained in $K(E[2])$, and again we have that $2\mid \gcd\big([K(E[p]):K],[K(E[q]):K]\big)$. 
\end{proof}

\begin{theorem}\label{infmanyfields} Let $\ell$ be a prime such that $\ell\mid\gcd\big([K(E[p]):K],[K(E[q]):K]\big)$. Then the set 
$$
\{L/K \mid T_{p,q}(E/L)=\Z/\ell\Z\}
$$
is infinite. 
\end{theorem}
\begin{proof} First, we claim that it suffices to prove that there exist subfields $F_1,F_2$, and $L$ of $K(E[pq])$ such that $F_1/K$ is Galois, $F_1$ and $F_2$ are entangleable by $L$, and $T_{p,q}(E/L)=\Z/\ell\Z.$ Assuming that this is true, let $L'/K$ be any field for which $LL'\cap F_1=L\cap F_1$. By Proposition \ref{inffield}, there are infinitely many such $L'$. By Proposition \ref{compent}, we know that $F_1$ and $F_2$ are entangleable by $LL'$ and 
\begin{equation}\label{galeq}
%\Gal(LL'F_1/LL')\cong\Gal(LF_1/L).
[LL'F_1:LL']=[LF_1:L].
\end{equation}
Now, since $F_1$ and $F_2$ are entangleable by $L$, Proposition \ref{entprop1}.(1) implies
$$
T_{p,q}(E/L)\cong \Gal(LF_1/L),
$$
and since $F_1$ and $F_2$ are also entangleable by $LL'$, again Proposition \ref{entprop1}.(1) implies
$$
T_{p,q}(E/LL')\cong \Gal(LL'F_1/LL').
$$
The result now follows from \eqref{galeq}. 

It remains to show the existence of the fields $F_1$, $F_2$, and $L$. If $\ell\mid [K_{p,q}:K]$ then there is a subfield $K\subseteq L\subseteq K_{p,q}$ such that $[K_{p,q}:L]=\ell$, hence $T_{p,q}(E/L)=\Z/\ell\Z$. In this case, take $F_1=F_2=K_{p,q}$ and note that $F_1$ and $F_2$ are entangleable by $L$ (see Remark \ref{selfent}). If $\ell\nmid [K_{p,q}:K]$, then by (the proof of) Theorem \ref{main1} there exist subfields $L\subseteq K(E[pq])$ and $F_2\subseteq K(E[q])$ such that $F_1=K(E[p])$ and $F_2$ are entangleable by $L$ and 
$
T_{p,q}(E/L)=\Z/\ell\Z.
$
\end{proof}

%$%%%%%%%%%%%%%%%%%%%%
\section{Classification of $(2,q)$-Entanglement Types}\label{classificationsection}

In this section we classify the possible $(p,q)$-entanglement types that can occur when $p=2$. 

\begin{theorem}\label{2qclassification} Fix an odd prime $q$ and let $E/\Q$ be an elliptic curve. Then 
$$
\Ent_{2,q}(E/\Q)\subseteq \{1,\Z/2\Z,\Z/3\Z, S_3\}.
$$
Furthermore, if $E$ has no $(2,q)$-entanglement over $\Q$ and $\rho_{E,q}$ is surjective  then 
$$
\Ent_{2,q}(E/\Q)=\begin{cases}\{1,\Z/2\Z,\Z/3\Z, S_3\}, &\text{if $\rho_{E,2}$ is surjective}\\
\{1,\Z/3\Z\}, &\text{if $\im \rho_{E,2}\cong \Z/3\Z$ }\\
\{1,\Z/2\Z\}, &\text{if $\im \rho_{E,2}\cong \Z/2\Z$ }\\
\{1\}, &\text{if $\rho_{E,2}$ is trivial.}
\end{cases}
$$
\end{theorem}
\begin{proof} 
Let $d=\gcd(\#\im\rho_{E,2},\#\im \rho_{E,q})$. Then 
$$
d\mid\gcd(\#\GL_2(\Z/2\Z),\#\GL_2(\Z/q\Z))=6.
$$
Proposition \ref{sizebound} therefore implies that the only possibilities for $(2,q)$-entanglement are
1, $\Z/2\Z$, $\Z/3\Z$, $S_3$ and $\Z/6\Z$. Type $\Z/6\Z$ cannot occur: if $T_{2,q}(E/L)=\Z/6\Z$ for some $L/\Q$ then, since $\Gal(L(E[2])/L)$ can be identified with the subgroup $\Gal(\Q(E[2])/\Q(E[2])\cap L)$  of $\Gal(\Q(E[2])/\Q)=S_3$, the inclusion $L\subseteq L_{2,q}\subseteq L(E[2])$ would force $S_3$ to have a subgroup with a $\Z/6\Z$ quotient, a contradiction. 

Now assume that $\rho_{E,q}$ is surjective and that $E$ has no $(2,q)$-entanglement over $\Q$. If $\rho_{E,2}$ is not surjective then $d=3,$ 2, or $1$, depending on whether the image of $\rho_{E,2}$ has order 3, 2, or 1, respectively. The result in these cases now follows from Theorem \ref{main1} and Proposition \ref{sizebound}.

If $\rho_{E,2}$ is surjective then $d=6$, so all $(2,q)$-entanglement types must have order dividing 6 by Proposition \ref{sizebound}. Clearly we can base extend to get trivial entanglement, and Theorem \ref{main1} yields entanglement types $\Z/2\Z$ and $\Z/3\Z$. It remains to show that $S_3\in \Ent_{2,q}(E/\Q)$. Since $\Q(E[pq])$ is the compositum of $\Q(E[2])$ and $\Q(E[q])$, and $E$ has no $(2,q)$-entanglement over $\Q$, we know that 
$$
\Gal(\Q(E[2q])/\Q)\cong\GL_2(\Z/2\Z)\times \GL_2(\Z/q\Z)
$$
By Lemma \ref{group-ent} (with $F=\Q(E[2q])$, $F_1=\Q(E[2])$, and $F_2=\Q(E[q])$) and Proposition \ref{entprop1}, it suffices to show that $\GL_2(\Z/2\Z)\times \GL_2(\Z/q\Z)$ contains a subgroup $H$ isomorphic to $S_3$ such that both $H\cap \GL_2(\Z/2\Z)\times 1$ and  $H\cap 1 \times \GL_2(\Z/q\Z)$ are trivial. This fact follows from Lemma \ref{S3lemma} below.
\end{proof}

\begin{lemma}\label{S3lemma} For all primes $q$, there exists a subgroup $H$ of $\GL_2(\Z/2\Z)\times \GL_2(\Z/q\Z)$ such that 
\begin{enumerate}
\item $H\cong S_3$, and 
\item $H$ intersects $\GL_2(\Z/2\Z)\times 1$ and $1\times \GL_2(\Z/q\Z)$ trivially.  
\end{enumerate}
\end{lemma}

\begin{proof} Fix generators $\sigma$ and $\tau$ of $S_3\cong  \GL_2(\Z/2\Z)$. The matrices
$$
\sigma' = \begin{pmatrix} 1 & 1 \\ 0& -1 \end{pmatrix} \quad \text{and}\quad
\tau'=\begin{pmatrix} -1 & 1 \\0 & 1 \end{pmatrix}
$$
generate a copy of $S_3$ inside $\GL_2(\Z/q\Z)$. (Both $\sigma'$ and $\tau'$ have order two, and their product has order 3.) The diagonal subgroup $H\subseteq\GL_2(\Z/2\Z)\times \GL_2(\Z/q\Z)$ generated by $(\sigma,\sigma')$ and $(\tau,\tau')$ now satisfies properties (1) and (2).
\end{proof}

\begin{example}\nf Let $E=X_0(11)$ (Cremona label \href{https://www.lmfdb.org/EllipticCurve/Q/11a1/}{\texttt{11a1}}). Using the LMFDB, one can check that $\rho_{E,2}$ is surjective and that the only primes $q<50$ which do not satisfy the hypotheses of Theorem \ref{2qclassification} are $q=5$ (exceptional) and $q=11$ (nontrivial $(2,11)$-entanglement). Thus
$$
\Ent_{2,q}(E/\Q)=\{1,\Z/2\Z,\Z/3\Z, S_3\}
$$
for all primes $q<50$ except $q=5,11$.
\end{example}

\begin{example}\nf Let $E/\Q$ be any elliptic curve with  no $(3,q)$-entanglement over $\Q$. If  $\rho_{E,3}$ and $\rho_{E,q}$ are surjective then 
$$
\Gal(\Q(E[3q])/\Q)\cong\GL_2(\Z/3\Z)\times \GL_2(\Z/q\Z).
$$
Using Lemma \ref{group-ent} (with $F=\Q(E[3q])$, $F_1=\Q(E[3])$ and $F_2=\Q(E[q])$) together with Proposition \ref{entprop1}(2), we use Magma 
%(similar to Examples \ref{S3} and \ref{S3S3ex}) 
to compute the following.
\begin{itemize}
\item For $q\in\{5,7,13\}$, $\Ent_{3,q}(E/\Q)$ contains the abelian groups $1$, $\Z/2\Z$, $\Z/3\Z$, $\Z/4\Z$, $\Z/2\Z\times \Z/2\Z$, $\Z/6\Z,$ $\Z/8\Z,$ and the nonabelian groups $S_3$, $D_4$, and $Q_8$.

\item  For $q\in\{11,17\}$, $\Ent_{3,q}(E/\Q)$ contains the abelian groups $1$, $\Z/2\Z$, $\Z/3\Z$, $\Z/4\Z$, $\Z/2\Z\times \Z/2\Z$, $\Z/6\Z,$ $\Z/8\Z,$ and the nonabelian groups $S_3$, $D_4$, $Q_8$, $D_6$, $SD_{16}$, $\SL_2(\Z/3\Z)$, and $\GL_2(\Z/3\Z)$.
\end{itemize}
\end{example}

\end{document}